\documentclass[a4paper, 11pt, reqno]{amsart}
\usepackage{amsmath,amsthm,amssymb,amscd}
\usepackage[english]{babel}
\usepackage[ansinew]{inputenc}
\usepackage[T1]{fontenc}
\usepackage{hyperref}

\usepackage[all]{xy}
\usepackage{xspace}
\usepackage{mathrsfs}
\usepackage[dvips]{color}
\usepackage{epsfig}
\usepackage{float}
\usepackage{wasysym}
\usepackage{setspace} 
\usepackage{enumerate}

\renewcommand{\sf}{\mathsf}

\newcommand{\diam}{\mathrm{diam}}
\newcommand{\Mod}{\mathrm{Mod}}

\newcommand{\R}{{\mathbb R}}

\newcommand{\be}{\begin{enumerate}}
\newcommand{\ee}{\end{enumerate}}

\numberwithin{figure}{section}

\newtheorem{theorem}{Theorem}[section]
\newtheorem{lemma}[theorem]{Lemma}
\newtheorem{proposition}[theorem]{Proposition}
\newtheorem{corollary}[theorem]{Corollary}
\newtheorem{definition}[theorem]{Definition}

\newenvironment{example}[1][Example.]{\begin{trivlist}
\item[\hskip \labelsep \textsc{ #1}]}{\end{trivlist}}
\newenvironment{remark}[1][Remark.]{\begin{trivlist}
\item[\hskip \labelsep \textsc{#1}]}{\end{trivlist}}

\makeatletter

\@addtoreset{equation}{section}
\makeatother

\title{Conformal dimension and canonical splittings of hyperbolic groups}
\author{Matias Carrasco Piaggio}
\email{matias.carrasco@math.u-psud.fr}
\address{Laboratoire de Math\'ematiques d'Orsay, Universit\'e Paris-Sud 11.}

\begin{document}

\maketitle
\begin{abstract}
We prove a general criterion for a metric space to have conformal dimension one. The conditions are stated in terms of the existence of enough local cut points in the space. We then apply this criterion to the boundaries of hyperbolic groups and show an interesting relationship between conformal dimension and some canonical splittings of the group.
\end{abstract}
\begin{quote}
\end{quote}
\begin{quote}
\footnotesize{\textsc{2010 Subject classification}: 30L10, 51F99, 20F67, 28A78.}
\end{quote}

\section{Introduction}
In this article we give sufficient conditions for a compact metric space $(X,d)$ to have conformal dimension one. The conformal dimension is a fundamental quasisymmetry invariant, introduced by Pansu in \cite{Pa}. Its original motivation is in the study of the quasiconformal structure of the boundary at infinity of a negatively curved space. For instance, any quasisymmetry invariant of the boundary of a hyperbolic group is a quasi-isometry invariant of the group. The understanding of conformal dimension has already given many applications in geometric group theory, in particular to the boundary characterization of Kleinian groups and to Cannon's conjecture \cite{BK1}. See also \cite{B,H2,H3,K,LP,MT} for other applications. 

There are different related versions of this invariant; in this article we are concerned with the \emph{Ahlfors regular conformal dimension}, a variant introduced by Bourdon and Pajot in \cite{BP}. It is defined by
$$\dim_{AR}(X,d):=\inf\left\{\dim_H(X,\theta):\theta\text{ is AR and }\theta\sim_{qs} d\right\},$$
where AR means Ahlfors regular, $\dim_H$ denotes Hausdorff dimension, and $\theta\sim_{qs} d$ means that $\theta$ is a distance on $X$ quasisymmetrically equivalent to $d$. That is, there exists an increasing homeomorphism $\eta:\R_+\to\R_+$ such that 
$$\frac{\theta(x,a)}{\theta(y,a)}\leq \eta\left(\frac{d(x,a)}{d(y,a)}\right),$$
for all distinct points $x,y,a\in X$. The conformal gauge of the metric $d$ is the quasisymmetric equivalence class of $d$. We recall that a distance $\theta$ on $X$ is Ahlfors regular of dimension $\alpha>0$ if there exists a Radon measure $\mu$ on $X$  and a constant $K\geq 1$ such that:
\begin{equation*}
r^\alpha K^{-1}\leq\mu\left(B_r\right)\leq Kr^\alpha,
\end{equation*}
for any ball $B_r$ of radius $0\leq r\leq\diam_{\theta} X$. In that case, $\mu$ is comparable to the $\alpha$-dimensional Hausdorff measure and $\alpha=\dim_H(X,\theta)$ is the Hausdorff dimension of $(X,\theta)$. We write $\dim_{AR} X$ when there is no ambiguity on the base metric $d$.

This kind of deformations can distort the Hausdorff dimension, and one can always quasisymmetrically deform the distance $d$ to obtain arbitrarily large Hausdorff dimension. The conformal dimension measures the best shape of $X$, and it is in general very difficult to compute. It is always bounded from below by the topological dimension $\dim_T X$. In particular, when the space is connected, $\dim_{AR} X\geq 1$ holds. 

This article deals with the problem of under which conditions we can quasisymmetrically deform the distance $d$ to obtain AR distances with Hausdorff dimension arbitrarily close to the topological dimension of the space. This question is of particular interest for the boundaries of hyperbolic groups. For instance, the AR conformal dimension of $\partial G$ is equal to the topological dimension $\dim_T\partial G=n\geq 1$ and is attained by a distance in the gauge if, and only if, $\partial G$ is quasisymmetrically equivalent to the Euclidean sphere $\mathbb{S}^n$ \cite{BK2}. For $n=1$, the same is true under the weaker hypothesis of $\partial G $ being homeomorphic to $\mathbb{S}^1$ \cite{CJ,Ga,Tu}. But the problem is far from being understood in the general case, even for low topological dimension. In this paper we address the case when $\dim_TX=1$.

Previously known non-trivial examples of spaces of conformal dimension one were very few, due to Bishop and Tyson \cite{BT}, Pansu (see \cite{BK1,Bu,CMT} for comments and generalizations), and Laakso (see \cite{TW} for generalizations). Although the techniques of proof are specific to each particular example, they are all related to the existence of local cut points. A point $x\in X$ is a \emph{local cut point} if there is a connected open set $x\in U\subset X$ such that $U\backslash\{x\}$ is not connected. The following condition gives a scale invariant bound on the amount of local cut points needed to disconnect the space into small pieces.

\begin{definition}[The UWS condition]\label{ubr}
We say that a connected and compact metric space $X$ has \emph{uniformly well spread} local cut points ---UWS for short--- if there exists a constant $C\geq 1$ such that for any point $x\in X$ and $r>0$, there is a finite set $P\subset B(x,r)$ verifying:
\begin{enumerate}
\item $\# P \leq C$, and
\item no connected component of $X\backslash P$ can intersect both $B\left(x,\frac{r}{2}\right)$ and $X\backslash\overline{B}(x, r)$.
\end{enumerate}
\end{definition}

We remark that one can always assume that the points of the subset $P$ are local cut points of $X$; this justifies the terminology. Recall that $X$ is a \emph{doubling} space if there exists a constant $N$ such that every ball can be covered using $N$ balls of half its radius; and, that $X$ is \emph{linearly connected} if there exists a constant $C$ such that for all $x,y\in X$, there exists a continuum $J$ containing $x$ and $y$ of diameter less than or equal to $Cd(x,y)$. This is also known as the \emph{bounded turning} condition. We obtain the following general criterion for conformal dimension one.

\begin{theorem}[General criterion for conformal dimension one]\label{tcriteregen}
Let $X$ be a doubling and compact metric space. If $X$ is linearly connected and satisfies the UWS condition, then the (AR) conformal dimension of $X$ is equal to one.
\end{theorem}

The main ingredient of the proof is an unpublished result of S. Keith and B. Kleiner stating an equality between the conformal dimension and the critical exponent associated to the combinatorial modulus (a proof of this result is given in \cite{CaPi}). A similar unpublished result to Theorem \ref{tcriteregen} was known by S. Keith and B. Kleiner; one of the motivations of the present article is to provide an accessible proof of this criterion and some of it consequences in the context of hyperbolic groups.

Note that Theorem \ref{tcriteregen} provides a large class of examples, including the examples mentioned before. A result of J. Mackay ensures that when the space $X$ verifies a quantitative analogue of the topological conditions of being locally connected and without local cut points, then the conformal dimension is greater than one \cite{M}. Both criteria provide a clear conceptual picture of the relationship between conformal dimension and local cut points.

The rest of the paper is devoted to the boundaries of hyperbolic groups. They are part of a larger class consisting of \emph{quasiselfsimilar} spaces (see Section \ref{groups}). For this class, the hypotheses of Theorem \ref{tcriteregen} are equivalent to the following topological conditions: the space $X$ is compact, connected, and verifies the \emph{well spread} local cut points condition ---WS for short: there exists a sequence of finite sets $P_n\subset X$ such that 
\begin{equation}
\delta_n=\sup\left\{\diam A:A\in\mathcal{C}_n\right\}\to 0,\text{ when } n\to+\infty,\tag{WS}
\end{equation}
where $\mathcal{C}_n$ denotes the set of connected components of $X\backslash P_n$. 

It is a remarkable fact that the topology of the boundary is reflected in the splittings of the group \cite{Bow}. For example, if $G$ is a one-ended hyperbolic group which is not virtually Fuchsian, then $\partial G$ has a local cut point if and only if $G$ splits over a virtually cyclic subgroup \cite{Bow}. This motivates the problem of characterizing the hyperbolic groups whose boundaries at infinity have conformal dimension one in terms of the properties of their canonical splittings. Recall that a splitting of $G$ is given by an action of $G$ on a simplicial tree $\sf{T}$, without edge inversions and with finite quotient \cite{DD}.

The candidate groups for $\dim_{AR}\partial G=\dim_T\partial G=1$ are essentially obtained as repetitive amalgamated products or HNN-extensions of virtually Fuchsian and finite groups over elementary subgroups. To see this, note that if $G$ is a hyperbolic group, we can decompose $G$ repeatedly over finite and virtually cyclic subgroups, until ---at least if there is no $2$-torsion--- all subgroups are finite, virtually Fuchsian, or one-ended without local cut points on the boundary \cite{DP,Va}. If $\partial G$ has no local cut points, then $\dim_{AR}\partial G> 1$ \cite{M}.\footnote{We remark that the boundary of a one-ended hyperbolic group is either homeomorphic to the Sierpi\'nski carpet or the Menger sponge if it has topological dimension one and no local cut points \cite{KaK}.} Therefore, the question can be formulated as follows: is it true that if $\partial H$ has a local cut point for every one-ended quasiconvex subgroup $H$ of $G$, then $\dim_{AR}\partial G=1$? The answer to this question seems to highly depend on how are embedded (in $G$) the one-ended quasiconvex subgroups of $G$.

We apply Bowditch's work \cite{Bow} on the structure of local cut points of the boundaries of one-ended hyperbolic groups, by relating the JSJ splitting of $G$ with the WS property, to deduce a partial answer. The following is our main result for hyperbolic groups:

\begin{theorem}[Characterization of the WS property for hyperbolic groups]\label{TEOBRGROUPES}
Let $G$ be a one-ended hyperbolic group. Then $\partial G$ satisfies the WS property if and only if either 
\begin{itemize}
\item[(i)] $\partial G$ is homeomorphic to the circle $\mathbb{S}^1$, or 
\item[(ii)] all the rigid type vertices in the JSJ splitting of $G$ are virtually free.
\end{itemize}
In particular, in this case we have $\dim_{AR}\partial G=1$.
\end{theorem}

We refer to Section \ref{groups} for precise definitions. In the case (i) above, the WS property is trivially verified, the interesting case is (ii), but we need to split the statement in two cases because the JSJ splitting is not well defined for virtually Fuchsian groups. Note that if $G$ satisfies the hypothesis (ii) of Theorem \ref{TEOBRGROUPES}, then the AR conformal dimension is never attained by a distance in the gauge. In terms of quasiconvex subgroups of $G$, the condition (ii) above is equivalent to the following: the limit set $\Lambda_H$ of every one-ended quasiconvex subgroup $H$ of $G$ can be separated in $\partial G$ by removing a pair of points (see Section \ref{groups}). 

To clarify the meaning of condition (ii) in Theorem \ref{TEOBRGROUPES}, we can give a local version of the WS property. Let $\mathcal{E}$ be the union of all limit sets of the stabilizers of the edges of the JSJ splitting of $G$. For each $x\in \partial G$, we define the set $\mathcal{E}(x)$ to be the set of points $y\in \partial G$ such that $x$ and $y$ are in the same connected component of $\partial G\setminus P$, for any finite subset $P$ of $\mathcal{E}$, not containing $x$ nor $y$. The WS condition is equivalent to the triviality of all these sets: $\mathcal{E}(x)=\{x\}$ for all $x\in \partial G$. Moreover, the proof of Theorem \ref{TEOBRGROUPES} shows that this is always the case, unless $x$ belongs to the limit set $\Lambda_v$ of a rigid type vertex $v$, in which case $\mathcal{E}(x)$ contains the connected component of $\Lambda_v$ containing $x$. This explains why the existence of a non-virtually free rigid type vertex in the JSJ splitting of $G$ prevents the boundary $\partial G$ to satisfy the WS condition: the points in a non-trivial component of the limit set of a rigid vertex cannot be separated by a finite number of points.

\begin{remark}[Remark 1.1.]
In order to rigorously justify our approach to the characterization of hyperbolic groups of conformal dimension one, we need to show that conformal dimension is stable under maximal splittings over finite subgroups, i.e. an action of $G$ on a simplicial tree $\sf{T}$ such that the stabilizers of the edges are finite and the stabilizers of the vertices have at most one end \cite{D,S}. We call such a splitting the DS splitting of $G$. The precise statement is the following: denote by $\{v_1,\ldots,v_n\}$ a set of representatives of the orbits of the vertices of $\sf{T}$, and by $G_{v_i}$ their respective stabilizers. Then either
\begin{enumerate}
\item[(i)] $\dim_{AR}\partial G=0$ if all the $G_{v_i}$ are finite, or
\item[(ii)] $\dim_{AR}\partial G= \max\left\{\dim_{AR}\partial G_{v_i}:G_{v_i}\text{ is infinite}\right\}$ otherwise.
\end{enumerate}
This fact is well known to specialists in the field but we were unable to find it in the literature. We mention that the techniques used here allows one to prove it and a proof can be found in \cite{CapiT}. This allows us to consider the one-ended case.
\end{remark}

This fact with Theorem \ref{TEOBRGROUPES} provide us with the following class of hyperbolic groups of conformal dimension one:

\begin{corollary}
Let $G$ be a non-virtually free hyperbolic group. Suppose that for any one-ended vertex stabilizer $H$ of the DS splitting of $G$ we have either
\begin{itemize}
\item[(i)] $\partial H$ is homeomorphic to $\mathbb{S}^1$, or 
\item[(ii)] all the rigid type vertices in the JSJ splitting of $H$ are virtually free,
\end{itemize}
then the AR conformal dimension of $\partial G$ is equal to one.
\end{corollary}

We end this introduction by showing how Theorem \ref{tcriteregen} and \ref{TEOBRGROUPES} can be used to provide new interesting computations of the conformal dimension. Indeed, one remarkable consequence of Theorem \ref{TEOBRGROUPES} is the existence of convex cocompact Kleinian groups for which the conformal dimension is different from a well known geometric invariant in the context of hyperbolic $3$-manifolds. More precisely, let $G$ be a one-ended convex cocompact Kleinian group (whose limit set is not the entire sphere) and let $M$ be the compact hyperbolizable $3$-manifold with boundary defined by $\left(\mathbb{H}^3\cup \Omega\right)/G$, where $\mathbb{H}^3$ denotes the real hyperbolic $3$-space and $\Omega$ is the discontinuity set of the action of $G$ on the Riemann sphere, so that $G\simeq \pi_1(M)$. Consider the collection of all complete hyperbolic $3$-manifolds $N$ which are homeomorphic to the interior of $M$. Each $N$ is uniformized by a Kleinian group $G_N$ so that $N$ is isometric to the quotient $\mathbb{H}^3/G_N$. Define $d(N)$ to be the Hausdorff dimension of the limit set of $G_N$ equipped with the Euclidean metric (note that this dimension only depends on $N$), and let
$$D(M):=\inf_N d(N).$$ 
Each $N$ can be approximated in the strong topology by a convex cocompact one \cite{NS,Oh}. Since $d(N)$ is continuous in the strong topology (see for example \cite{CT}), we always have $\dim_{AR}\partial \pi_1(M)\leq D(M)$. In \cite{CMT} the authors completely characterize the compact hyperbolizable $3$-manifolds for which $D(M)=1$. This is the case if and only if $M$ is a \emph{generalized book of I-bundles}, see also \cite{CMT} for the definition. When $M$ is a generalized book of I-bundles, there is no rigid vertex in the JSJ decomposition of $\pi_1(M)$. And therefore we obtain the following:

\begin{corollary}[$\dim_{AR}\partial \pi_1(M)<D(M)$]\label{Dvsdimconf}
Let M be a compact hyperbolizable 3-manifold as above. If all the rigid vertices of the JSJ decomposition of $\pi_1(M)$ are virtually free, and there exists at least one rigid vertex, then $\dim_{AR}\partial \pi_1(M)<D(M)$.
\end{corollary}
That is, for such a group there is a definite gap between the infimum of Hausdorff dimensions of Kleinian deformations and the conformal dimension. See Section \ref{examplegroups} for a concrete example. The inequality $D(M)>1$ is a (not direct) consequence of Thurston's relative compactness theorem. It would be interesting to know if the conformal dimension of such a group can be obtained as the infimum of Hausdorff dimensions of the limit sets of convex cocompact actions of $G$ on CAT($-1$) spaces. See \cite[Thm 1.2]{Bou} for a similar type of result for hyperbolic buildings. 



\begin{remark}[Remark 1.2.]
Many of the ideas of this paper can also be applied to the study of conformal dimension for the repellors of dynamical systems induced by a certain class of expanding branched coverings, i.e. the topologically cxc maps, see \cite{HP1} for a precise definition. A dynamical sufficient condition for conformal dimension one also holds in this context. Nevertheless, the tools involved are quite different, so it will be explained elsewhere \cite{CaPi2}.
\end{remark}

\medskip
\textbf{Acknowledgments.} The author would like to thank Peter Haïssinsky for all his help and advice. He also thanks J. Alonso, M. Bourdon, B. Kleiner, C. Lecuire, J. Mackay and A. Sambarino for helpful discussions, comments and suggestions.

\subsection*{Outline of the paper}
In Section \ref{gen}, we prove Theorem \ref{tcriteregen}. The rest of the paper is devoted to hyperbolic groups. In Section \ref{groups} we recall Bowditch's work on the structure of local cut points on the boundaries of one-ended hyperbolic groups and the JSJ decomposition. Some basic estimates are given in Section \ref{usdef} and \ref{estdiam}. We then prove the key Proposition \ref{smallpartition} on virtually free groups in Section \ref{virtuallyfree}. We deduce Theorem \ref{TEOBRGROUPES} in Section \ref{theproof}. We end by showing how to construct explicit examples of groups verifying the hypotheses of Corollary \ref{Dvsdimconf} in Section \ref{examplegroups}.
\section{Proof of Theorem \ref{tcriteregen}}\label{gen}

The main ingredient of the proof is the combinatorial modulus, we refer to \cite{BK,Can,H1} for a general treatment. Using an appropriate sequence of finite coverings of $X$ whose mesh tends to zero, we can define combinatorial versions of conformal moduli, from which we are able to compute the conformal dimension of the space. See Remark 2.1 at the end of this section for a heuristic idea of the proof.

We start by remarking that the UWS property is equivalent to the following: there exists a function $C:(0,1)\to\R_+$ such that for all $x\in X$ and $0<s<r\leq\diam X$, there is a finite set $P\subset B(x, r)$ of cardinality bounded from above by $C(s/r)$, and such that no connected component of $X\backslash P$ can intersect both $B\left(x,s\right)$ and $X\backslash\overline{B}(x,r)$. In fact, suppose the UWS condition is satisfied, and take $0<s<r\leq\diam X$. Let 
$$n:=\left[\frac{2}{r-s}\right]+1, \text{ and }\epsilon=\frac{1}{4n}.$$
Consider the compact set $K=\overline{B}\left(x,r-n^{-1}\right)\setminus B\left(x,s+n^{-1}\right)$. For every $y\in K$, the ball $B(y,2\epsilon)$ is contained in $\overline{B}(x,r)\setminus B(x,s)$. By the doubling condition, we can cover $K$ by less than $M$ balls $B(y_i,\epsilon)$, with $y_i\in K$. Since $\diam K\leq 2r$, the constant $M$ depends only on
$$\frac{\epsilon}{r}\asymp \frac{1}{nr}\asymp \frac{1-s/r}{2}.$$
For each center $y_i\in K$, consider a set $P_i\subset B(y_i,2\epsilon)$, with $\# P_i\leq C$, given by the UWS condition. Then it suffices to take $P=\bigcup_iP_i$, which is contained in $B(x,r)$ and of cardinal number less than or equal to $M\cdot C$.

Fix $a>1$ a big enough constant and consider $\{X_i\}$ a sequence of maximal $a^{-i}$-separated sets. By maximality, the balls $B(x, a^{-i})$, where $x \in X_i$, define a covering $\mathcal{S}_i$ of $X$. We write $\mathcal{S}:=\bigcup_k\mathcal{S}_k$. The combinatorial modulus is defined as follows. Let $n,k\geq 1$, for every ball $B\in\mathcal {S}_k$, we consider the family of curves $\Gamma(B)$ in $X$ that ``join'' the ball $B$ with the complement of the ball $2B$, i.e. $\gamma\cap B\neq \emptyset$ and $\gamma\cap X \setminus 2B \neq \emptyset$. Here $2B$ denotes the ball with the same center as $B$ and twice its radius. Given $p>0$, we define the $p$-combinatorial modulus of the ``annulus'' $(B, 2B)$ at relative scale $n$ by
\begin{equation*} \label{pmodboule}
\Mod_p\left(\Gamma(B),\mathcal{S}_{k+n}\right): = \inf\limits_{\rho}\mathrm{Vol}_p(\rho),\text{ where }\mathrm{Vol}_p(\rho):=\sum\limits_{A\in\mathcal{S}_{k + n}}\rho\left(A\right)^p,
\end{equation*}
and where the infimum is taken over all weight functions $\rho: \mathcal{S}_{k + n} \to \R_+$ which are $\Gamma(B)$-admissible, i.e. for any curve $\gamma\in\Gamma(B)$, we have
\begin{equation*}\label{longcomb}
\ell_\rho\left(\gamma\right)=\sum\limits_{A\cap \gamma\neq\emptyset}\rho\left(A\right)\geq 1.
\end{equation*}
In the same way one defines the $p$-combinatorial modulus $\Mod_p(\Gamma,\mathcal{S}_n)$ for any curve family $\Gamma$ in $X$. In the sequel, we write $M_{p,n}(B)$ instead of $\Mod_p\left(\Gamma(B),\mathcal{S}_{k+n}\right)$. Therefore, for each $p> 0$, we obtain a sequence $\{M_{p,n}\}_n$, where $M_ {p, n}$ is the $p$-combinatorial modulus of $X$ at scale $n$:
$$M_{p,n}:=\sup\limits_{B\in \mathcal{S}}M_{p,n}(B).$$
In other words, the modulus $M_{p, n}$ takes into account all the ``annuli'' of $X$ with a fixed radius ratio equal to $2$. Relevant to the proof is the asymptotic behavior of the sequence $\{M_{p,n}\}_n$, and its dependence on $p$. For fixed $p$, $\{M_ {p,n}\}_n$ verifies a sub-multiplicative inequality \cite{CaPi}, and it is therefore natural to consider the \emph{critical exponent} ${\sf p}_c$ defined by
$${\sf p}_c:=\inf\left\{p> 0:\liminf_n M_{p,n}=0\right\}.$$
We derive Theorem \ref{tcriteregen} from the following result.

\begin{theorem}[\cite{CaPi} Thm 3.12]\label{theocritexp} 
Let $X$ be a doubling, linearly connected, compact metric space. Then ${\sf p}_c$ is equal to the AR conformal dimension of $X$.
\end{theorem}

In the sequel we prove that ${\sf p}_c=1$. For this, we will prove that $M_{p,n}\lesssim\eta_n^{p-1}$ for all $p>1$, where $\eta_n$ is a sequence of positive real numbers that tends to zero as $n\to\infty$. 

Let $x\in X_k$ and write $r=a^{-k}$. Recall that $\Gamma (B)$ is the family of curves of $X$ joining $B(x,r)$ and $X\setminus B\left(x, 2r\right)$. We denote by $P=P(B)\subset 3/2\cdot B$ the finite set given by the UWS property, which verifies that $\gamma\cap P\neq\emptyset$ for any $\gamma\in\Gamma(B)$, and that $\#P\leq C$, where $C$ is a uniform constant. We have $\Gamma(B)\subset \bigcup_{z\in P}\Gamma_z$, where $\Gamma_z$ is the family of curves $\gamma$ in $X$ such that $z\in\gamma$ and $\gamma\cap X\setminus B(z,s)$, where $s=r/3$. In particular, we obtain
$$M_{p,n}(B)\leq \sum\limits_{z\in P}\Mod_p\left(\Gamma_z,\mathcal{S}_{k+n}\right)\leq C\max\limits_{z\in P}\left\{\Mod_p\left(\Gamma_z,\mathcal{S}_{k+n}\right)\right\}.$$
We must bound from above the combinatorial modulus of $\Gamma_z$. We take $m\geq 1$, and for $i=0,\ldots m-1$, we set
$$A_i=\overline{B\left(z,2^{-i}s\right)}\setminus B\left(z,2^{-(i+1)}s\right).$$
By the UWS property, for each $z\in P$ and each $i=0,\ldots,m-1$, there exists a finite set $R_{z,i}\subset A_i$, with cardinal number bounded from above by a universal constant $K$, such that any curve $\gamma$ of $X$ verifying $\gamma\cap B(z,2^{-(i+1)}s)\neq\emptyset$ and $\gamma\cap B(z,2^{-i}s)\neq\emptyset$, must pass by $R_{z,i}$. For $n\geq 1$ such that $a^{-(k+n)}\leq 2^{-m}s$, consider the set
$$U=\left\{A\in\mathcal{S}_{k+n}: A\cap R\neq\emptyset\right\},$$
where
$$R:=\bigcup_{i=0}^{m-1}R_{z,i}.$$
We define $\rho:\mathcal{S}_{k+n}\to\R_ +$ by
\begin{equation}\label{idea}\rho(A):=\begin{cases}\frac{1}{m} &\text{if } A\in U\\
0 &\text{otherwise.}
\end{cases}\end{equation}
On the one hand, since any curve $\gamma\in\Gamma_z$ must cross each $A_i$ for $i=0,\ldots,m-1$, we have that $\rho$ is $\Gamma_z$-admissible. On the other hand, there exists a constant $M$, which depends only on the doubling constant of $X$ and $a$, such that 
$$\#U\leq M\cdot\#\left\{\bigcup_{i=0}^{m-1}R_{z,i}\right\}\leq (MK)\cdot m:=K'm.$$ 
Therefore,
$$\mathrm{Vol}_p(\rho)=\frac{\#U}{m^p}\leq \frac{K'}{m^{p-1}}.$$
This shows that $\Mod_p\left(\Gamma_z,\mathcal{S}_{k+n}\right)\leq K'm^{1-p}$. Thus, it suffices to take
$$\eta_n:=\left[\frac{n\log a-\log 3}{\log 2}\right]^{-1}.$$
This ends the proof of Theorem \ref{tcriteregen}.

\begin{figure}
\centering
\setlength{\unitlength}{1cm}
\begin{picture}(11,5)
\includegraphics[width=11cm]{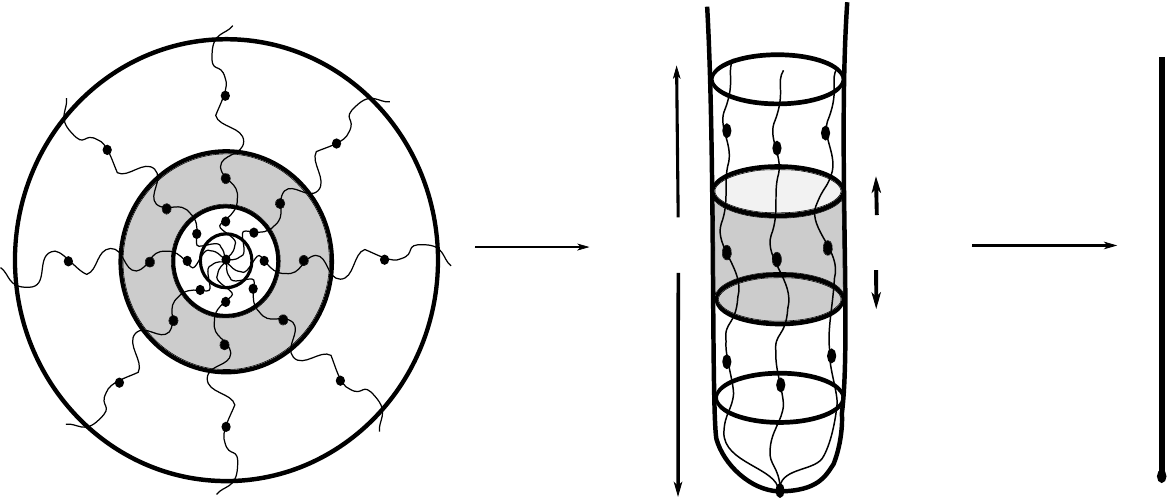}
\put(-2.9,2.35){\footnotesize $s'/m$}
\put(-4.7,2.3){\footnotesize $s'$}
\put(-1.8,2){\footnotesize $m\to +\infty$}
\put(-6.5,2){\footnotesize $d\mapsto \theta$}
\put(-0.3,0){\footnotesize $z$}
\put(-8,0){\footnotesize $B(z,s)$}
\put(-7.8,2.6){\footnotesize $A_i$}
\end{picture}
\caption{Idea of the proof of Theorem \ref{tcriteregen}. The points represent the set $R$ given by the UWS condition applied to each annulus $A_i$. Some curves passing through the annuli are drawn.\label{fig:idea}}
\end{figure}
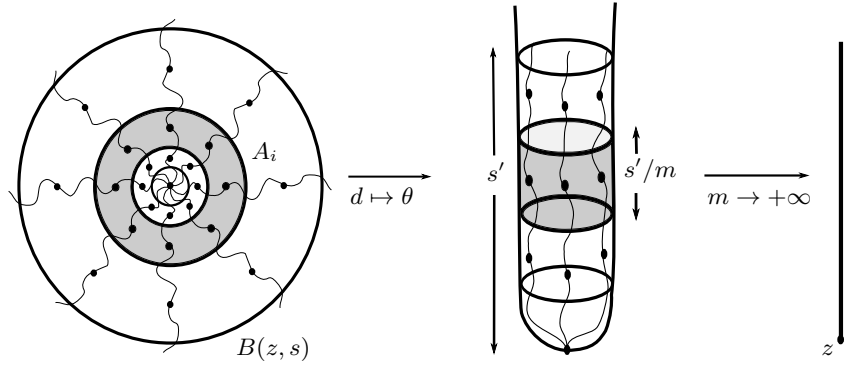

\begin{remark}[Remark 2.1.]
We give an heuristic idea of what is the shape of the quasisymmetric deformations of the original distance $d$ on $X$ induced by the combinatorial modulus. Intuitively, the dimension of a distance $\theta$ on $X$ is $\leq p$ if the cardinal of a maximal $\epsilon r$-separated set in a ball of radius $r$ is bounded from above by $C\epsilon^{-p}$. Using the same notations as in the proof, the weights defined by (\ref{idea}) produce a distance $\theta$ for which the balls of the covering $\mathcal{S}_{k+n}$ which do not intersect $R$ have a very small diameter, and the balls in $U$ have diameter almost equal to $1/m$ times the diameter $s'$ of $B(z,s)$. This implies that the set of centers of the balls in $U$ (i.e. intersecting the the set $R$) becomes an almost $s'/m$-separated set in $B(z,s)$ for the distance $\theta$. Its cardinal number is bounded from above by $K'm$. This shows why the dimension of $\theta$ is close to one, and tends to one when $m\to +\infty$. See Figure \ref{fig:idea}.
\end{remark}

\begin{remark}[Remark 2.2]
The UWS condition fails when there is an infinite number of definite diameter pairwise disjoint curves in the space. Therefore, it is not difficult to find examples of spaces of conformal dimension equal to $1$, but not verifying the UWS property. However, the LC and UWS assumptions are in some sense optimal. One can construct an example of a compact connected space which verifies WS and LC, but not UWS, and has AR conformal dimension strictly bigger than $1$; and, also a space which verifies UWS, but not LC, and of AR conformal dimension strictly bigger than $1$. See \cite[Chapter 5]{CapiT}.
\end{remark}

\section{The WS property for hyperbolic groups}\label{groups}

In this section we give the proof of Theorem \ref{TEOBRGROUPES}. Let $G$ be a Gromov-hyperbolic group and fix $\sf{C}$ a Cayley graph of $G$ with respect to some finite generating set $S$. Equip $\sf{C}$ with the length distance which makes each edge of $\sf{C}$ isometric to the unit interval. We can identify the boundary of $G$ with $\partial \sf{C}$. In the sequel, we fix a visual distance $\theta$ on $\partial \sf{C}$ (of visual parameter $\epsilon>0$ and base point at the identity). This metric is Ahlfors regular of dimension $h/\epsilon$ whenever $G$ is non-elementary; here $h$ is the volume entropy of the action of $G$ on $\sf{C}$ \cite{Coo}. In particular, $\left(\partial \sf{C},\theta\right)$ is a doubling uniformly perfect compact metric space. Moreover, $\partial \sf{C}$ is a quasiselfsimilar space \cite[Prop. 4.6]{H2}, this is also known as the Sullivan \emph{conformal elevator principle} \cite{Sul}. If $G$ is one-ended, then $\partial C$ is automatically locally connected, and by self-similarity it verifies LC \cite{BK4}. Therefore, for the boundary of a one-ended hyperbolic group, the hypotheses of Theorem \ref{tcriteregen} are equivalent to the WS property.

Let $G$ be a one-ended hyperbolic group which is not a cocompact virtually Fuchsian group, and denote by ${\sf T}$ the Bass-Serre tree associated to the JSJ splitting of $G$ (see Section \ref{JSJ} for the precise definition). In the sequel, we denote by $\sf{T}_0$ and by $\sf{T}_1$ the vertices and the edges of ${\sf T}$. We equip $\sf{T}$ with the length distance that makes each edge isometric to the unit interval. Given $v\in \sf{T}_0$ and $e\in \sf{T}_1$, we denote by $G_v$ and by $G_e$ the stabilizers of $v$ and $e$ respectively. We also let $\Lambda_v$ and $\Lambda_e$ be the limit set of $G_v$ and $G_e$ respectively. Thus, if $e\in \sf{T}_1$ and $v,w\in \sf{T}_0$ are the endpoints of $e$, then $G_e=G_v\cap G_w$. 

The main result of Section \ref{virtuallyfree} is Proposition \ref{smallpartition}. This is the key step in the proof of Theorem \ref{TEOBRGROUPES}. We prove that if $v$ is a virtually free vertex of the JSJ splitting of $G$, then for a given scale $\delta>0$, we can find a finite set $\sf{F}_v$ of edges of $\sf{T}$ incident to $v$, such that: 
\begin{itemize}
\item removing their limit sets $\Lambda_e$ from $\Lambda_v$ produces a partition of $\Lambda_v$ into small pieces ($\diam\leq\delta$), and
\item with the property that if $e$ is an edge incident to $v$ not in $\sf{F}_v$, then $\Lambda_e$ is contained in exactly one piece of this partition; i.e. $\Lambda_e$ does not ``connect'' different pieces. 
\end{itemize}
This consist in an inductive step of the proof.  Section \ref{theproof} contains the proof of the theorem, which is summarized in Proposition \ref{brgroupes}. The idea is to construct a big enough finite subtree $\sf{T}_\delta$ of $\sf{T}$, such that removing the limit sets of the vertices of $\sf{T}_\delta$ leaves only pieces of diameter less than $\delta$. Then we can apply the previous lemma to decompose each one of the vertices of $\sf{T}_\delta$ in a coherent fashion, producing the desired partition of the boundary of $G$. We invite the reader to look at the example in Section \ref{examplegroups}, and also Example 3.1, before going into the details of the proof.

\subsection{The JSJ splitting and local cut points}\label{JSJ}

The aim of this section is to recall the basic properties of Bowditch's JSJ splitting for hyperbolic groups. By a \emph{virtually Fuchsian} group we mean a non-elementary hyperbolic group that acts properly discontinuously and by isometries on the real hyperbolic plane $\mathbb{H}^2$. We say that the group is \emph{cocompact} virtually Fuchsian if the action is. The action is not necessarily faithful, but its kernel is finite. If $H$ is a virtually Fuchsian group, we say that $H$ is \emph{convex cocompact} if the action is cocompact on its convex core. Recall that its convex core is the minimal $H$-invariant closed convex subset of $\mathbb{H}^2$. In this case, the \emph{peripheral} subgroups of $H$ are the stabilizers of the boundary connected components of its convex core.

Bowditch's JSJ decomposition theorem \cite[Thm 0.1 and Thm 5.28]{Bow} asserts that there exists a minimal simplicial action of $G$ on a simplicial tree $\sf{T}$, without edge inversions and with finite quotient $\sf{T}/G$, whose edge stabilizers are virtually cyclic, and the vertices of $\sf{T}$ are of three types:
\medskip
\begin{itemize}
\item[($\sf{T}_C$)] \emph{Virtually cyclic}: the stabilizer is a maximal virtually cyclic subgroup of $G$. Its valence in $\sf{T}$ is finite and at least two. The limit set $\Lambda_v$ corresponds to a pair of local cut points of $\partial\sf{C}$.
\medskip
\item[($\sf{T}_S$)] \emph{Surface or MHF}: the stabilizer is a quasiconvex, non-elementary virtually free convex cocompact Fuchsian subgroup of $G$. Its peripheral groups are precisely the stabilizers of its incident edges. They are maximal in the following sense: if $H$ occurs as the vertex group of a finite splitting of $G$ over virtually cyclic subgroups, in such a way that $H$ admits a convex cocompact Fuchsian action on $\mathbb{H}^2$ whose peripheral groups are the stabilizers of its incident edges, then $H$ is contained in a subgroup of surface type.
\medskip
\item[($\sf{T}_R$)] \emph{Rigid}: the stabilizer is quasiconvex, non-elementary and not of Surface type, and admits no decomposition over virtually cyclic subgroups relative to the stabilizers of its incident edges. The limit set of a rigid type vertex is characterized by the following property: $\Lambda_v$ is a maximal closed subset of $\partial\sf{C}$ with the property that it cannot be separated by two points of $\partial \sf{C}$. 
\end{itemize}

We refer to \cite{Sho,BH} for the definition and basic properties of quasiconvex subgroups. Partition the vertices of the JSJ tree $\sf{T}$ as $\sf{T}_C\cup \sf{T}_S\cup \sf{T}_R$, which are respectively the vertices of virtually cyclic, surface and rigid type. These types are mutually exclusive, are preserved by the action of $G$, and two adjacent vertices are never of the same type. The splitting is quasi-isometry invariant and the maximality of the splitting can be formulated as follows: any local cut point of $\partial \sf{C}$ is contained in the limit set of the stabilizer of a vertex in $\sf{T}_C\cup\sf{T}_S$. 

To construct the simplicial tree $\sf{T}$, Bowditch uses the structure of local cut points of $\partial \sf{C}$. The limit set $\Lambda_e$ of $e\in  \sf{T}_1$ consists of exactly two points which are fixed points of any element of infinite order in $G_e$. The limit set $\Lambda_e$ separates the boundary of $\sf{C}$, and the two points of $\Lambda_e$ are local cut points. There are only a finite number of connected components of $\partial \sf{C}\setminus\Lambda_e$, and every component has $\Lambda_e$ as frontier.

The vertices of surface type $\sf{T}_S$ are virtually free (with a peripheral structure) and the limit set $\Lambda_v$ is a Cantor set of local cut points. The stabilizer of $\Lambda_v$, for the action of $G$ on $\partial \sf{C}$, is the same as the stabilizer of $v$ for the action of $G$ on $\sf{T}$, i.e. $G_v$. Let us explain in more detail the natural bijection between the edges of $\sf{T}$ incident to $v$, and the peripheral subgroups of $G_v$. The limit set of the peripheral subgroups are the ``jumps'' of the Cantor set $\Lambda_v$, which can be defined using the embedding of $\Lambda_v$ into $\mathbb{S}^1$ given by the action of $G_v$ on $\mathbb{H}^2$. Denote by $J_v$ the set of jumps. For $\tau=\{x,y\}\in J_v$, denote by $\mathcal{C}_\tau$ the set of connected components of $\partial \sf{C}\setminus \tau$. Then the set of connected components $\mathcal{C}_v$ of $\partial \sf{C}\setminus \Lambda_v$ is given by
\begin{equation}\label{compcantor}\mathcal{C}_v=\bigcup_{\tau\in J_v}\mathcal{C}_\tau.\end{equation}
This means that two points are separated by $\Lambda_v$ if and only if they are separated by a jump. Therefore, in order to study the WS property it is enough to consider the local cut points which are in the limit sets of the edges of $\sf{T}$:
$$\mathcal{E}:=\bigcup_{e\in \sf{T}_1}\Lambda_e.$$

\begin{example}[Example 3.1]
Let us finish our discussion of the JSJ splitting by giving a simple example of a rigid type vertex. Consider two isometric copies of a closed hyperbolic surface $S$ and $S'$. Let $\gamma$ be a closed geodesic in $S$, and suppose that $\gamma$ is filling; i.e. the connected components of $S\setminus \gamma$ are simply connected. Denote by $\gamma'$ the copy of $\gamma$ in $S'$. Consider the complex 
$$X=S\sqcup S'/\gamma\sim\gamma',$$
and let $G=\pi_1(X)$. Then the quotient graph $\sf{T}/G$ of the JSJ splitting of $G$ has three vertices: $\langle\gamma\rangle\in\sf{T}_C$, $\pi_1(S)\in\sf{T}_R$ and $\pi_1(S')\in\sf{T}_R$. To see this, note that since $\gamma$ is a filling geodesic, given any to points $x,y\in \Lambda_{\pi_1(S)}$, there exists a conjugate of $\langle\gamma\rangle$ whose fixed points $\{u,v\}$ separate $\{x,y\}$ in $\Lambda_{\pi_1(S)}$. Since $\partial \sf{C}$ contains a copy of $\Lambda_{\pi_1(S')}$ passing through $\{u,v\}$, we see that $\Lambda_{\pi_1(S)}$ cannot be separated in $\partial \sf{C}$ by a pair of points. Note that $\partial\sf{C}$ does not satisfy the WS property.
\end{example}

\subsubsection{Some useful definitions}\label{usdef}

The following definition will be important in relating the action of $G$ on $\sf{T}$ and the boundary of $\sf{C}$. Fix an arbitrary vertex $v_0\in \sf{T}_0$. We define an equivariant Lipschitz continuous function $\phi: \sf{C}\to \sf{T}$ by setting $\phi(1)=v_0$ and $\phi(g)=g\cdot v_0$ for all $g\in G$. This determines $\phi $ on $G$, the vertices of $\sf{C}$. If $\sf{I}=(g,h)$ is an edge of $\sf{C}$, we extend $\phi$ so that it sends $\sf{I}$ linearly to the unique arc of $\sf{T}$ that joins $\phi(g)$ and $\phi(h)$. The image of an edge of $\sf{C}$ is either a vertex or an arc of $\sf{T}$. Thus, by minimality $\phi$ is surjective.

\noindent Let us introduce some notation:

\begin{itemize}
\item If $A$ is a closed subset of $\sf{C}$, we denote by $\Lambda_A$ the set $\overline{A}\setminus A$, where the closure is taken in the space $\sf{C}\cup \partial \sf{C}$. 
\item Let $e\in \sf{T}_1$, we denote by $m_e$ its mid-point and $\sf{M}_e:=\phi^{-1}(m_e)$.
\item For $v\in  \sf{T}_0$, let $\sf{T}^v_e$ be the closure of the connected component of $\sf{T}\setminus\{m_e\}$ which does not contain $v$, and let $\sf{C}^v_e=\phi ^{-1}\left(\sf{T}^v_e\right)$. For $v=v_0$, we simply write $\sf{T}_e$ and $\sf{C}_e$; we also denote by $\sf{T}^-_e$ the the other component of $\sf{T}\setminus\{m_e\}$ and $\sf{C}^-_e=\phi ^{-1}(\sf{T}^-_e)$.

\item We denote by $\sf{E}_v$ the set of edges of $\sf{T}$ which are incident to $v$.
\item For a vertex $v\in \sf{T}_0$, let $\sf{S}_v$ be the star-shaped connected subset of $\sf{T}$ formed by $v$ and all the segments joining $v$ to $m_e$ with $e\in \sf{E}_v$. Let $\sf{M}_v:=\phi^{-1}(\sf{S}_v)$. Note that $\sf{M}_e$ and $\sf{M}_v$ are respectively $G_e$ and $G_v$ invariant, and for every edge $e\in \sf{E}_v$, we have $\sf{M}_e\subset \sf{M}_v$.
\item For every vertex $v\neq v_0$ of $\sf{T}$, there exists exactly one edge $e^*\in \sf{E}_v$ such that $\sf{C}^v_{e^*}=\sf{C}^-_{e^*}$, i.e. such that $\sf{C}^v_{e^*}$ contains $v_0$. Let $\sf{E}^*_v=\sf{E}_v\setminus\{e^*\}$.
\end{itemize}

\noindent Finally, we say that a function $f:A\to[0,\infty)$ defined on a set $A$, tends to infinity, $f(a)\to\infty$ for $a\in A$, if for all $N>0$, the set $\{a\in A: f(a)\leq N\}$ is finite.

\subsubsection{A general estimate for the diameter of $\Lambda_{\sf{C}_e}$}\label{estdiam}

The estimates that follow are essentially proved in \cite{Bow,Bow2}. Our aim is to deduce some topological properties of $\partial \sf{C}$ from the action of $G$ on $\sf{T}$. These estimates are true in general for any quasiconvex splitting of $G$. We will also use them in the next section for splittings over finite subgroups.

Note that up to conjugacy there is only a finite number of groups $G_e$, but that the entire family of edge stabilizers $\{G_e:e\in {\sf T}_1\}$ is not uniformly quasiconvex. Nevertheless, for each edge $e\in  \sf{T}_1$, the set $\sf{M}_e$ is at finite Hausdorff distance from $G_e$. Moreover, the sets $\{\sf{M}_e: e\in  \sf{T}_1\}$ are uniformly quasiconvex (say of constant $K_{qc}$) and therefore
\begin{equation}\label{estidiamqe}
\diam\ \Lambda_e=\diam\ \Lambda_{\sf{M}_e}\lesssim\exp\left(-\epsilon\mathrm{dist}\left(1,\sf{M}_e\right)\right).
\end{equation}
This allows to prove the following \cite[Section 1]{Bow}. 

\begin{lemma}\label{estiqvze}
The sets $\{\sf{C}_e: e\in  \sf{T}_1\}$ are uniformly quasiconvex of constant $K_{qc}$. In particular,
\begin{equation}\label{estidiamze}
\diam\ \Lambda_{\sf{C}_e}\lesssim\exp\left(-\epsilon\mathrm{dist}\left(1,\sf{C}_e\right)\right).
\end{equation}
Moreover,
\begin{equation}\label{distzealuno}
\mathrm{dist}\left (1,\sf{C}_e\right)=\mathrm{dist}\left(1,\sf{M}_e\right)\to\infty\text{ for } e\in  \sf{T}_1.
\end{equation}
\end{lemma}

\begin{remark}[Remark 3.1]
It also holds that $\sf{M}_v$ is at finite Hausdorff distance from $G_v$. Moreover, the set $\{\sf{M}_v:v\in  \sf{T}_0\}$ is a locally finite cover of $\sf{C}$ by $K_{qc}$-quasiconvex sets, and $\Lambda_{\sf{M}_v}=\Lambda_v$.
\end{remark}

\subsection{The key step of the proof}\label{virtuallyfree}

The purpose of this section is to prove a key fact on virtually free groups, Proposition \ref{smallpartition}, for the proof of Theorem \ref{TEOBRGROUPES}. We suppose in this section  that $G$ is a virtually free group. By \cite{D,S} the group $G$ acts on a simplicial tree $\sf{Z}$, so that the edge and vertex stabilizers are all finite subgroups of $G$. We use the same notation as in the previous section, in particular, we consider the equivariant map $\phi:{\sf C}\to {\sf Z}$ as before. In this case, there is a bijection between $\partial \sf{C}$ and $\partial \sf{Z}$; i.e. for every geodesic ray $\gamma$ in $\sf{C}$, the image $\phi(\gamma)$ is unbounded in $\sf{Z}$. Let $\partial \sf{Z}^\pm_e$ be the set of rays of $\sf{Z}$ which are eventually contained in $\sf{Z}^\pm_e$. Note that in this case $\Lambda_e=\emptyset$ for any edge $e\in  \sf{Z}_1$, so the sets $\Lambda_{\sf{C}_e}$ are open and closed in $\partial \sf{C}$.

We need one more estimate in this particular case. For any element $g\in G$ we set $|g|=d(1, g)$. Let $\sf{E}_0=\{e_1,\ldots, e_r\}$ be a set of representatives of edge orbits. For each edge $e\in \sf{Z}_1$ we choose, and fix in the sequel, $g_e\in G$ such that $e=g_e\cdot e_i$. The number $|g_e|$ does not depend, up to an additive constant, on the choice of $g_e$, and in fact 
\begin{equation}\label{diamzege}
\mathrm{dist}(1,\sf{M}_e)=|g_e|+O(1), \ \forall\,e\in\sf{Z}_1.
\end{equation}
It is not difficult to see that any geodesic asymptotic $\Lambda_{\sf{C}_e}$ and to $\Lambda_{\sf{C}^-_e}$ must pass uniformly close from $\sf{M}_e$. This fact together with the uniform perfectness of $\partial \sf{C}$ implies
\begin{equation}\label{distentrelesdz2}
\mathrm{dist}(\Lambda_{\sf{C}_e},\Lambda_{\sf{C}^-_e})\asymp\exp(-\epsilon|g_e|)\ \forall\, e\in \sf{Z}_1.
\end{equation}

\begin{lemma}\label{estimsommetfini}
Let $v$ be a vertex of $\sf{Z}_0$. Then for every pair of edges $e,e'\in \sf{E}^*_v$, we have
\begin{equation}\label{eqestimsommetfini}
\mathrm{dist}\left(\Lambda_{\sf{C}^v_e},\Lambda_{\sf{C}^v_{e'}}\right)\asymp\exp(-\epsilon|g_e|).
\end{equation}
\end{lemma}
\begin{proof}
We know from (\ref{diamzege}), that $|g_e|=|g_{e'}|+O(1)$. Moreover, since $\Lambda_{\sf{C^v_{e'}}}$ is contained in the complement of $\Lambda_{\sf{C}^v_e}$, by (\ref{distentrelesdz2}) we have
$$\mathrm{dist}\left(\Lambda_{\sf{C}^v_e},\Lambda_{\sf{C}^v_{e'}}\right)\gtrsim\exp(-\epsilon|g_e|).$$
The upper bound follows from the hyperbolicity of $\sf{C}$, the fact that $\sf{M}_e\cup \sf{M}_{e'}\subset \sf{M}_v$ (which is contained in a ball of uniform radius), and the fact that any  ray from $1$ asymptotic $\Lambda_{\sf{C}^v_e}$, and any ray from $1$ asymptotic $\Lambda_{\sf{C}^v_{e'}}$, must pass through $\sf{M}_e$ and $\sf{M}_{e'}$ respectively.
\end{proof}

The following two claims will be important for the proof of Proposition \ref{smallpartition}: 
\medskip

\noindent \emph{Claim 1}: Let $x$ and $y$ be two different points on $\partial \sf{C}$, and suppose that the stabilizer $\mathrm{Stab}\left(\{x,y\}\right)$ is infinite. Then for any fixed $\delta>0$, there exists a finite set $\{g_1,\ldots,g_N\}\subset G$ such that if $g\in G$ verifies $\theta(g\cdot x,g\cdot y)\geq \delta$, then it can be written as 
\begin{equation*}\label{finite}
g=g_ih,\text{ for some } i\in\{1,\ldots,N\}\text{ and }h\in \mathrm{Stab}\left(\{x,y\}\right).
\end{equation*}

For the second claim, we need more preparation. Let $\mathcal{H}:=\{H_1,\ldots,H_k\}$ be a finite collection of two-ended subgroups of $G$. For each $i\in \{1,\ldots,k\}$ the limit set $\Lambda_i$ of $H_i$ consists of exactly two points of $\partial \sf{C}$, which we denote by $h_i^+$ and $h_i^-$. Define the decomposition space associated to $\mathcal{H}$ by $\mathrm{D}:=\partial \sf{C}/\sim$, where two points $x$ and $y$ of $\partial \sf{C}$ are identified if, either $x=y$, or there exists $g\in G$ such that $\{x,y\}=g\cdot \Lambda_i$. We denote by $\mathrm{p}:\partial \sf{C}\to \mathrm{D}$ the projection defined by $\sim$ and we equip $\mathrm{D}$ with the quotient topology. Note that any point of $\mathrm{D}$ has at most two pre-images under $\mathrm{p}$. By \emph{Claim 1}, $\mathrm{D}$ is a compact Hausdorff topological space. In the sequel, we will always suppose that $\mathrm{D}$ is connected. This implies
\medskip

\noindent \emph{Claim 2}: For any edge $e\in \sf{Z}_1$, there exists $i\in\{1,\ldots,k\}$ and $g\in G$ such that $g\cdot h_i^\pm\in\Lambda_{\sf{C}_e}$ and $g\cdot h_i^\mp\in\Lambda_{\sf{C}^-_e}$.
\medskip

This allows to consider, for any edge $e$ of $\sf{Z}$ the nonempty set 
$$\mathcal{P}_e:=\left\{g\in G:\exists\ i\in\{1,\ldots,k\}\text{ s.t. } g\cdot h_i^\pm\in\Lambda_{\sf{C}_e},g\cdot h_i^\mp\in \Lambda_{\sf{C}^-_e}\right\}.$$
Let 
$$P_e:=\bigcup_{g\in\mathcal{P}_e}g\cdot \Lambda_i,$$ and let $\mathrm{D}_e=\mathrm{p}\left(P_e\right)$ be its projection onto $\mathrm{D}$.

\begin{lemma}\label{finiteedgeset}
For any edge $e$ of $\sf{Z}$, the set $P_e$ is finite and $\mathrm{D}\setminus\mathrm{D}_e$ is disconnected.
\end{lemma}
\begin{proof}
Fix $e$ an edge of $\sf{Z}$, we first prove that $P_e$ is finite. Indeed, if $g\in \mathcal{P}_e$, then by (\ref{distentrelesdz2})
$$\theta(g\cdot h_i^\pm,g\cdot h_i^\mp)\geq \mathrm{dist}(\Lambda_{\sf{C}_e},\Lambda_{\sf{C}^-_e}):=c_e>0,$$ where $c_e$ only depends on $e$. By \emph{Claim 1}, there exists a finite subset $F_e$ of $G$ such that $g=g'h$, with $g'\in F_e$ and $h\in \mathrm{Stab}(\Lambda_i)$. Therefore, $g\cdot \Lambda_i=g'\cdot\Lambda_i$, and $P_e$ is finite. 

Let us show that $\mathrm{D}_e$ is a cut-set in $\mathrm{D}$. Since $P_e$ is finite, it is closed in $\partial \sf{C}$. Since $\mathrm{p}(\Lambda_{\sf{C}_e})\cap\mathrm{p}(\Lambda_{\sf{C}^-_e})=\mathrm{D}_e$, we see that $\Lambda_{\sf{C}_e}\setminus P_e$ and $\Lambda_{\sf{C}^-_e}\setminus P_e$ are open and closed saturated sets in $\partial \sf{C}$. Taking the projection of these sets by $\mathrm{p}$ gives a non-trivial partition of $\mathrm{D}\setminus\mathrm{D}_e$.  
\end{proof}

\begin{remark}[Remark 3.2]
It is important to note that $P_e$ is finite because $\Lambda_{e}=\emptyset$. The finiteness conclusion of Lemma \ref{finiteedgeset} does not hold in general if the edge stabilizers are not finite. This is a key point for the proof of Theorem \ref{TEOBRGROUPES}.
\end{remark}

We prove now the main result of this section.
\begin{proposition}\label{smallpartition}
Let $\delta>0$. There exists a finite set $\sf{E}_\delta$ of edges of $\sf{Z}$ such that if we denote by $\mathrm{D}_\delta:=\bigcup_{e\in \sf{E}_\delta}\mathrm{D}_e$ and by $P_\delta:=\mathrm{p}^{-1}(\mathrm{D}_\delta)$, then $\partial \sf{C}\setminus P_\delta$ can be partitioned into a finite union of open and closed saturated sets $A$ with $\mathrm{diam} A\leq\delta$. In particular, if $\mathrm{U}$ is a connected component of $\mathrm{D}\setminus \mathrm{D}_\delta$, then 
$$\mathrm{diam}\ \mathrm{p}^{-1}(\mathrm{U})\leq \delta.$$
Moreover, this implies that if for some $g\in G$ and $i\in \{1,\ldots,k\}$ we have $g\cdot \Lambda_i\subset \partial \sf{C}\setminus P_\delta$, then $g\cdot \Lambda_i\subset A$ for some subset $A$ of the partition.
\end{proposition}
\begin{proof}
Let $\eta>0$ and let $\sf{E}_{\eta,0}$ be the finite subset of edges of $\sf{Z}$ such that $\diam\Lambda_{\sf{C}_e}\geq \eta$, and let $\sf{Z}_\eta$ be its convex hull in $\sf{Z}$, so that $\sf{Z}_\eta$ is a finite subtree of $\sf{Z}$. Let $\sf{E}_\eta$ be the set of edges of $\sf{Z}_\eta$. For each vertex $v\in \sf{Z}_\eta$ we denote by $\sf{E}^*_v(\eta):=\sf{E}_v\setminus \sf{E}_\eta$. We can assume without loss of generality that $v_0\in \sf{Z}_\eta$, so that for each vertex $v$ of $\sf{Z}_\eta$, and each edge $e\in \sf{E}^*_v(\eta)$, we have $\Lambda_{\sf{C}^v_e}=\Lambda_{\sf{C}_e}$, and in particular, that $\mathrm{diam}\ \Lambda_{\sf{C}^v_e}\leq\eta$. We set 
$$\mathrm{D}_\eta:=\bigcup\limits_{e\in \sf{E}_\eta}\mathrm{D}_e.$$
If $v_1$ and $v_2$ are two different vertices of $\sf{Z}_\eta$, and if $e_j\in \sf{E}^*_{v_j}(\eta),\ j=1,2$, then
\begin{equation}\label{intersection}
\mathrm{p}\left(\Lambda_{\sf{C}^{v_1}_{e_1}}\right)\cap\mathrm{p}\left(\Lambda_{\sf{C}^{v_2}_{e_2}}\right)\subset \mathrm{D}_\eta.
\end{equation}
Let $P_\eta:=\mathrm{p}^{-1}(\mathrm{D}_\eta)$. For each vertex $v$ of $\sf{Z}_\eta$, let $A_v$ be the union of the sets $\Lambda_{\sf{C}^v_e}\setminus P_\eta$ with $e\in \sf{E}_v^*(\eta)$. Then $A_v$ is open and closed in $\partial \sf{C}\setminus P_\eta$, and is a saturated set. This provides a partition of $\partial \sf{C}\setminus P_\eta$ by open and closed saturated sets. Therefore, if $\mathrm{U}$ is a connected component of $\mathrm{D}\setminus \mathrm{D}_\eta$, there exists a vertex $v$ of $\sf{Z}_\eta$ such that $\mathrm{U}\subset \mathrm{p}(A_v)$. Since $A_v$ is saturated, this implies that $\mathrm{p}^{-1}(\mathrm{U})\subset A_v$.

\noindent We end by estimating the diameter of the sets $A_v$. Indeed, by the definition of $\sf{Z}_\eta$ and by Lemma \ref{estimsommetfini}, we have
\begin{equation*}
\mathrm{diam} A_v  \leq 2\max\limits_{e\in \sf{E}_v^*(\eta)}\{\mathrm{diam} \Lambda_{\sf{C}^v_e}\}+\max\limits_{e,e'\in \sf{E}_v^*(\eta)}\mathrm{dist}\left(\Lambda_{\sf{C}^v_e},\Lambda_{\sf{C}^v_{e'}}\right)\leq K\eta,
\end{equation*}
where $K$ is a uniform constant. So it is enough to take $\eta$ small enough so that $K\eta\leq \delta$.
\end{proof}

\begin{remark}[Remark 3.3]
We use the fact that $G$ is virtually free in the following key point of the proof of Proposition \ref{smallpartition}. The sets $A_v$ introduced in the proof do not contain the limit sets $\Lambda_v$ for a vertex $v\in \sf{Z}_\eta$. When $G$ is virtually free this is not a problem since $\Lambda_v=\emptyset$. Thus, the sets $A_v$ form a covering of $\partial \sf{C}\setminus P_\eta$ (note also that $\partial \sf{C}$ can be identified with $\partial\sf{T}$ when $G$ is virtually free). Equation (\ref{intersection}) is always true by construction, so in general one needs to add $\Lambda_v\setminus P_\eta$ to $A_v$ in order to obtain a partition of $\partial\sf{C}\setminus P_\eta$. This prevents the diameter of $A_v$ to be small. Nevertheless, since the diameter of $\mathrm{p}^{-1}(x)$ tends to zero for $x\in\mathrm{D}$, there is only a finite number of classes ``connecting'' two given limit sets $\Lambda_v$ and $\Lambda_w$. Therefore, adding a finite number of points to $P_\eta$ if necessary, one can separate $\mathrm{p}(\Lambda_v)$ and $\mathrm{p}(\Lambda_w)$ by removing a finite number of points of $\mathrm{D}$.
\end{remark}

\begin{remark}[Remark 3.4]
We will use in the next section the fact that the set $P_\delta$ in the statement of Proposition \ref{smallpartition} can be assumed to contain any given finite set of equivalence classes. 
\end{remark}

\subsection{Proof of Theorem \ref{TEOBRGROUPES}}\label{theproof}

\subsubsection{All rigid type vertices are virtually free implies WS\label{directo}}
Assume now that in the JSJ splitting of $G$ all rigid type vertices are virtually free. In this case, all the vertices of $\sf{T}$ are either non-elementary virtually free or virtually cyclic. 

Let $v$ be a vertex in $\sf{T}_S\cup \sf{T}_R$, so $G_v$ is non-elementary virtually free. Recall that $\sf{E}_v$ is the set of edges of $\sf{T}$ which are incident to $v$. Since the action of $G$ on $\sf{T}$ preserves the type of a vertex and two adjacent vertices are of different type, the stabilizer of $\sf{E}_v$ is also $G_v$. This implies that the number of $G$-orbits of edges incident to $v$ is the same as the number of $G_v$-orbits, and therefore $\sf{E}_v/G_v$ is finite. Let $\{f_1,\ldots,f_{k_v}\}$ be a finite set of representatives of $\sf{E}_v/G_v$, and write $H_j(v)$ the two-ended subgroup $G(f_j)$ of $G_v$. We consider the decomposition space $\mathrm{D}_v$ associated to $v$ and we denote by $\mathrm{p}_v:\Lambda_v\to\mathrm{D}_v$ the quotient map like in Section \ref{virtuallyfree}. Another equivalent way to construct $\mathrm{D}_v$ is by identifying the points of $\partial \sf{C}$ which are in the same $\Lambda_{\sf{C}^v_e}$ for $e\in \sf{E}_v$. Since $\partial \sf{C}$ is connected, the same is true for $\mathrm{D}_v$. Moreover, the space $\mathrm{D}_v$ is either homeomorphic to the circle if $v\in \sf{T}_S$, or it is locally connected without cut-points and cut-pairs if $v\in \sf{T}_R$ \cite{CaMa,H3}. Therefore, we can apply the results of Section \ref{virtuallyfree}.

The subsets used in Section \ref{virtuallyfree} to separate the decomposition space $\mathrm{D}_v$ were constructed by projecting a finite union of sets of the form $g\cdot \Lambda_{f_j}$ with $g\in G_v$ and $j\in\{1,\ldots,k_v\}$. Then, in this case, they are the projection of a finite union of $\Lambda_e$ with $e\in \sf{E}_v$. We can state Proposition \ref{smallpartition} in the following way: for all $\delta>0$, there exists a finite set of edges $\{e_1,\ldots,e_n\}$, where $n$ depends on $v$ and $\delta$, of $\sf{E}_v$ such that if we set $P_v:=\bigcup_{i=1}^n\Lambda_{e_i}$, then
$$\Lambda_v\setminus P_v=\bigcup_{i=1}^{m}A_i,$$
where the sets $A_i$ are disjoint, open and closed in $\Lambda_v\setminus P_v$, and saturated sets. Here $m$ also depends on $v$ and $\delta$. Moreover, for any other edge $e\in \sf{E}_v$, there exists $i\in\{1,\ldots,m\}$ such that $\Lambda_e\subset A_i$. Given such a partition, we denote by $\sf{E}_v(i)$, for $i=1,\ldots,m$, the edges $e$ incident to $v$ such that $\Lambda_e\subset A_i$. Naturally,
$$\sf{E}_v=\{e_1,\ldots, e_n\}\cup\bigcup_{i=1}^{m}\sf{E}_v(i),$$
is a partition of all edges incident to $v$. Let
$$K_i(v)=A_i\cup\bigcup\limits_{e\in \sf{E}_v(i)}\Lambda_{\sf{C}^v_e}\text{ for } i=1,\ldots, m.$$

\begin{lemma}\label{Leski}
The sets $\{K_i(v)\}_{i=1}^{m}$ are pairwise disjoint and closed in $\partial \sf{C}\setminus P_v$.
\end{lemma}
The proof is similar to that of Lemma 4.6 in \cite{Bow2}. We can partition the boundary of $\sf{C}$ as
\begin{equation}\label{decompbord}
\partial \sf{C}=\bigcup_{i=1}^n\left(\Lambda_{\sf{C}^v_{e_i}}\setminus P_v\right)\cup \bigcup_{i=1}^{m}K_i(v)\cup P_v,
\end{equation}
and the sets of the first and second union are closed in $\partial \sf{C}\setminus P_v$. Furthermore, we have
\begin{equation}\label{diamki}
\diam\ K_i(v)\leq\diam\ A_i+2\cdot\max\limits_{e\in \sf{E}_v(i)}\left\{\diam\ \Lambda_{\sf{C}^v_e}\right\},
\end{equation}
Summarizing, we obtain the following.
\begin{lemma}\label{lemdiamki}
Let $v\in \sf{T}_S\cup \sf{T}_R$ be a vertex of $\sf{T}$. Then for any $\delta> 0$, there exists a finite set of edges $e_1,\ldots,e_n\in \sf{E}_v$ such that $\diam\ K_i(v)\leq\delta$ for all $i=1,\ldots,m$.
\end{lemma}
\begin{proof}
Let $\delta> 0$. There is a finite subset $\sf{E}_1\subset \sf{E}_v$, such that if $e\in \sf{E}_v\setminus \sf{E}_1$, the diameter of $\Lambda_{\sf{C}^v_e}$ is bounded from above by $\delta/4$. Also by the previous paragraph, there exists a finite subset $\sf{E}_2=\{e_1,\ldots, e_n\}\subset \sf{E}_v$, which we can assume containing $\sf{E}_1$, such that $\diam A_i\leq\delta/2$. Therefore, from (\ref{diamki}), we have $\diam\ K_i(v)\leq\delta$.
\end{proof}

We now finish the proof of the WS property.

\begin{proposition}\label{brgroupes}
Let $G$ be a one-ended hyperbolic group which is not a cocompact virtually Fuchsian group. Suppose that in the JSJ splitting of $G$ all rigid type vertices are virtually free. Then for any $\delta> 0$, there exists a finite set $\{e_1,\ldots, e_N\}$ of edges of $\sf{T}$, such that
$$\max\left\{\diam\ U: U\in\mathcal{U}\right\}\leq\delta,$$
where $\mathcal{U}$ is the set of connected components of $\partial \sf{C}\setminus P$, and $P=\bigcup_{i=1}^N\Lambda_{e_i}$. In particular, $\partial \sf{C}$ satisfies the WS property.
\end{proposition}
\begin{proof}
We chose $v_0$ in the definition $\phi$ to be any vertex in $\sf{T}_S\cup \sf{T}_R$. Let $\delta>0$, and consider the set $\sf{E}_\delta$ of edges $e\in \sf{T}_1$ such that $\diam\ \Lambda_{\sf{C}_e}\geq\delta$. By Lemma \ref{estiqvze}, $\sf{E}_\delta$ is a finite set. Let $\sf{V}_\delta$ be the union of $\{v_0\}$ and the vertices of the edges in $\sf{E}_\delta$. Let $\sf{T}_\delta$ be the convex hull of $\sf{V}_\delta$. Then $\sf{T}_\delta$ is a finite subtree of $\sf{T}$. We can suppose that for every vertex $w\in \sf{T}_C\cap \sf{T}_\delta$, all the edges incident to $w$ are in $\sf{T}_\delta$, because the valence of $w$ in $\sf{T}$ is finite. In this way, all the terminal vertices of $\sf{T}_\delta $ are in $\sf{T}_S\cup \sf{T}_R$. We will prove that it is enough to take $P$ to be the union of the limit sets of a finite collection of edges belonging to a slightly bigger set than $\sf{T}_\delta$. Let $v$ be a vertex in $\sf{T}_\delta$. We need to consider four cases: 
\medskip

\noindent $\bullet$ \emph{Case 1:} Suppose $v=v_0$. By Lemma \ref{lemdiamki}, there is 
$$\sf{F}_{v_0}:=\{e_1(v_0),\ldots, e_{n_{v_0}}(v_0)\}\subset \sf{E}_{v_0},$$
such that
\begin{equation}\label{delta1}
\diam\ K_i\left(v_0\right)\leq\delta,\text{ for all } i=1,\ldots, m_{v_0}.
\end{equation}
We can assume that the edges of $\sf{T}_\delta$ incident to $v_0$ belong to $\sf{F}_{v_0}$. Denote by $P_{v_0}$ the union of $\Lambda_{e_i(v_0)}$ for $i=1,\ldots, n_{v_0}$. Therefore, $\partial \sf{C}$ can be partitioned as in (\ref{decompbord}).
\medskip

\noindent $\bullet$ \emph{Case 2:} Suppose $v\neq v_0$ is any non-terminal vertex in $(\sf{T}_S\cup \sf{T}_R)\cap \sf{T}_\delta$. There is exactly one edge $e^*$ in $\sf{E}_v$ such that $\sf{T}^v_{e^*}=\sf{T}^-_{e^*}$. It necessarily belongs to $\sf{T}_\delta$ because $v_0\in \sf{T}_\delta$. For the others, we have $\sf{T}^v_e=\sf{T}_e$. By Lemma \ref{lemdiamki}, there are edges 
$$\sf{F}_v:=\{e_1(v),\ldots, e_{n_v}(v)\}\subset \sf{E}_v,$$
such that
\begin{equation}\label{delta1}
\diam\ K_i\left(v\right)\leq\delta,\text{ for all } i=1,\ldots, m_v.
\end{equation}
We can assume that the edges of $\sf{T}_\delta$ incident to $v$ belong to $\sf{F}_v$. Denote by $P_v$ the union of $\Lambda_{e_i(v)}$ for $i=1,\ldots, n_v$. Therefore, $\Lambda_{\sf{C_{e^*}}}$ is partitioned as
\begin{equation}\label{estar}
\Lambda_{\sf{C}_{e^*}}=\bigcup_{\sf{F}_v\setminus\{e^*\}}\left(\Lambda_{\sf{C}_{e_j(v)}}\setminus P_v\right)\cup\bigcup_{i=1}^{m_v}K_i(v)\cup P_v.
\end{equation}

\noindent $\bullet$ \emph{Case 3:} Suppose that $v\in \sf{T}_C\cap \sf{T}_\delta$. Then all the edges incident to $v$ are in $\sf{T}_\delta$. Let $e^*$ be the edge incident to $v$ such that $\sf{T}^v_{e^*}=\sf{T}^-_{e^*}$. If we write the other edges incident to $v$ as $\sf{E}^*_v=\{e_i=(v,v_i): i=1,\ldots, n_v\}$, then
\begin{equation}\label{pasoind}
\Lambda_{\sf{C}_{e^*}}=\bigcup_{i=1}^{n_v}\Lambda_{\sf{C}_{e_i}}=\bigcup_{i=1}^{n_v}\Lambda_{\sf{C}^v_{e_i}}.
\end{equation}
The intersection of two sets of this decomposition is equal to $\Lambda_v=\Lambda_{e^*}=\Lambda_{e_i}$, for $i=1,\ldots, n_v$. This allows us to replace $\Lambda_{\sf{C}_{e^*}}$ by the union of $\Lambda_{\sf{C}_{e_i}}$, where the $v_i$ all belong to $(\sf{T}_S\cup \sf{T}_R)\cap \sf{T}_\delta$.
\medskip

\noindent $\bullet$ \emph{Case 4:} Suppose that $v$ is a terminal vertex of $\sf{T}_\delta$. Let as before $e^*$ be the edge incident to $v$ such that $\sf{T}^v_{e^*}=\sf{T}^-_{e^*}$, which in this case is the unique edge of $\sf{T}_\delta$ incident to $v$. Then $\sf{E}^*_v\subset  \sf{T}_1\setminus T_\delta$, and by construction for any $e\in E^*_v$, we have $\diam\ \Lambda_{\sf{C}^v_e}=\diam\ \Lambda_{\sf{C}_e}\leq\delta$. Then the set $\Lambda_{\sf{C}_{e^*}}$ can be decomposed as in (\ref{estar}), and all the sets in the first and second union are of diameter less than or equal to $\delta$.

Starting from $v_0$, a repetitive combination of the partitions obtained in Cases 1 to 4 give us the set $P$. Indeed, we obtain a partition $\partial \sf{C}=A\cup B \cup P$, where
\begin{equation}\label{decompbr}
A:=\bigcup\limits_{v\in (\sf{T}_S\cup \sf{T}_R)\cap \sf{T}_\delta}\bigcup_{i=1}^{m_v} K_i(v),\ B:=\bigcup_{v\in\mathrm{ter}(\sf{T}_\delta)}\bigcup\limits_{e\in \sf{E}^*_v}(\Lambda_{\sf{C}_e}\setminus P_{v}),
\end{equation}
where $\mathrm{ter}(\sf{T}_\delta)$ denote the set of terminal vertices of $\sf{T}_\delta$, and
\begin{equation}\label{P}
P:=\bigcup\limits_{v\in (\sf{T}_S\cup \sf{T}_R)\cap T_\delta}P_v\end{equation}
The sets in the decomposition (\ref{decompbr}) defining $A$ and $B$ are pairwise disjoint and closed in $\partial \sf{C}\setminus P$, and of diameter less than or equal to $\delta$. This implies that any connected component of $\partial \sf{C}\setminus P$ is contained in one of these sets. This ends the proof.
\end{proof}

\subsubsection{WS is not satisfied when a non-virtually free rigid type vertex exists} Suppose there exists $v\in \sf{T}_R$ such that $G_v$ is not virtually free. First, decompose $G_v$ over finite groups, i.e. the DS splitting of $G_v$. Note that points in different connected components of $\Lambda_v$ can be separated by removing a finite number of local cut points of $\partial \sf{C}$, see Remark 3.3. Consider a vertex $w$ of the DS splitting of $G_v$ whose stabilizer $G_w$ is one-ended. We will show that $\Lambda_w$ cannot be separated by a finite set of local cut points of $\partial \sf{C}$.

Suppose that $g$ is a loxodromic element of $G_w$, and denote by $\Lambda_g=\{g^+,g^-\}$ its fixed points. The limit set $\Lambda_w$ is locally connected and without \emph{global} cut points, so $\Lambda_w\setminus \Lambda_g$ has a finite number of connected components; which we denote by $\Omega_i,\ i=1,\ldots,n$, $n\geq 1$. Note that since the infinite cyclic group $\langle g\rangle$ permutes the components $\{\Omega_i\}$, up to taking a big enough power of $g$, we can suppose that each $\Omega_i$ is fixed by $g$. The set $\Lambda_w\setminus \Lambda_g$ is contained in a connected component $\Omega$ of $\partial \sf{C}\setminus\Lambda_g$, because $v$ is a rigid vertex of $G$. We suppose that $n\geq 2$. For $i\in\{1,\ldots,n-1\}$, let $\alpha_{i}:[0,1]\to \Omega$ be a curve with $\alpha_i(0)\in\Omega_i$ and $\alpha_i(1)\in \Omega_{i+1}$. For each $i$, we connect $\alpha_i(1)$ with $\alpha_{i+1}(0)$ by a curve in $\Omega_{i+1}$, in order to obtain a curve $\alpha:[0,1]\to\Omega$ which intersects each $\Omega_i$. Moreover, given any $\delta>0$, there exists $k\geq 0$ such that the curves $\alpha^+_g:=g^k\cdot \alpha$ and $\alpha^-_g:=g^{-k}\cdot \alpha$ are respectively contained in the balls $B(g^+,\delta)\setminus \{g^+\}$ and $B(g^-,\delta)\setminus\{g^-\}$. In summary, this shows that all the components $\{\Omega_i\}_{i=1}^n$ of $\Lambda_w\setminus \Lambda_g$ can be connected by two curves of $\partial \sf{C}$, which are contained in arbitrarily small punctured balls centered at the fixed points of $g$. We need to consider two cases:
\medskip

\noindent $\bullet$ \emph{Case 1}: Suppose first that $\Lambda_w$ is not homeomorphic to the circle. Consider the JSJ decomposition of $G_w$ with associated tree $\sf{T}_w$. If it is trivial, we are done because there are no local cut points in $\Lambda_w$. If not, it is enough to consider the local cut points of $\Lambda_w$ which are in the limit set of an edge $e$ of $\sf{T}_w$, see (\ref{compcantor}).  Note that $\Lambda_e=\{g^+,g^-\}$ is the set of fixed points of a loxodromic $g$.

Let $\{e_1,\ldots,e_m\}$ be any finite subset of edges of $\sf{T}_w$. Consider $\delta>0$ small enough so that the balls centered at $P:=\bigcup_{k=1}^m\Lambda_{e_k}$ and of radius $\delta$ are disjoint. Write $\Lambda_{e_k}:=\{g_k^+,g_k^-\}$ and for each $k$ consider the curves $\alpha_k^\pm$ contained in $B(g_k^\pm,\delta)\setminus\{g_k^\pm\}$ as before. Since the balls are disjoint, all the curves $\{\alpha_k^\pm\}_{k=1}^m$ avoid the set $P$. Then the set 
$$(\Lambda_w\setminus P)\cup\bigcup_{k=1}^m (\alpha_k^+\cup\alpha_k^-)$$ is a connected subset of $\partial \sf{C}\setminus P$ containing $\Lambda_w\setminus P$. Since $\diam (\Lambda_w\setminus P)=\diam\Lambda_w$, $\partial \sf{C}$ does not satisfies the WS condition.
\medskip

\noindent $\bullet$ \emph{Case 2}: Suppose now that $\Lambda_w$ is homeomorphic to the circle. Then any pair of distinct points $\{x,y\}$ of $\Lambda_w$ is a cut-pair, and $\Lambda_w\setminus\{x,y\}=\Omega_1\cup\Omega_2.$ Take a curve $\alpha$ connecting $\Omega_1$ and $\Omega_2$ as before. Given $\delta>0$, there exists a loxodromic element of $G_w$ such that $g^+\in B(x,\delta)$ and $g^-\in B(y,\delta)$. Therefore, applying exactly the same argument as in Case 1 we show that $\Lambda_w\setminus P$ is contained in a connected component of $\partial \sf{C}\setminus P$ for any finite subset $P$ of $\Lambda_w$.

\subsection{A simple example with $\dim_{AR}\partial \pi_1(M)<D(M)$}\label{examplegroups}

In this section we give a simple example of a one-ended convex cocompact Kleinian group $G=\pi_1(M)$ for which the conclusion of Corollary \ref{Dvsdimconf} holds; namely, that $\dim_{AR}\partial \pi_1(M)<D(M)$, where $M$ is the hyperbolizable $3$-manifold with boundary whose interior is isometric to $\mathbb{H}^3/G$.

Recall that $M$ is called a generalized book of $I$-bundles if one may find a disjoint collection $A$ of essential annuli in $M$ such that each component $R$ of the manifold obtained by cutting $M$ along $A$ is either a solid torus or homeomorphic to an $I$-bundle such that $\partial R\cap \partial M$ is the associated $\partial I$-bundle (they correspond to the surface type vertices in the JSJ decomposition of $G$). In particular, there is no rigid type vertex in the JSJ decomposition of $G$.

Let $H$ be a handlebody with boundary $S$ a closed surface of genus $g\geq 2$, and let $\Gamma=\{\gamma_1,\ldots,\gamma_n\}$ be a multicurve on $S$; i.e. the elements of $\Gamma$ are disjoint simple closed curves on $S$. Consider a small smooth neighborhood $A$ of $\Gamma$ in $S$ so that $A=\{A_1,\ldots,A_n\}$ is a collection of disjoint annuli in $S$.

Let $T_i$ for $i=1,\ldots,n$ be $n$ copies of the solid torus $D^2\times S^1$. On the boundary $S_i$ of $T_i$ let $A_{T,i}$ be a smooth annulus whose core curve generates the $\pi_1(T_i)$. We glue together each solid torus $T_i$ to $H$ identifying $A_i$ with $A_{T,i}$. Denote by $R$ the $3$-manifold obtained.

For each $i$, let $B_i$ be a collection of $m_i\geq 1$ disjoint smooth annuli in $S_i$ for which each core curve generates $\pi_1(T_i)$. Suppose that $B_i$ is also disjoint from $A_{T,i}$. Denote by $m=m_1+\cdots +m_n$. Let $F$ be an $I$-bundle (over a surface which is not necessarily connected) with vertical boundary consisting of exactly $m$ smooth annuli. Glue $F$ to $R$ along the annuli $B_i$, $i=1,\ldots,n$, and call $M$ the $3$-manifold with boundary obtained. Let $G=\pi_1(M)$. 

Suppose now that $\Gamma$ is chosen so that the boundary of any properly embedded disk or annulus in $H$ intersects at least one of the curves in $\Gamma$. Then $G$ is one-ended, and $\pi_1(H)$ does not split over a cyclic subgroup relative to the subgroups $\{\langle \gamma_i\rangle\}_{i=1}^n$, i.e. $\pi_1(H)$ is a rigid type vertex in the JSJ decomposition of $G$. Indeed, the vertices of the JSJ decomposition of $G$ are given by $\{\langle \gamma_i\rangle\}_{i=1}^n\in \sf{T}_C$, $\{\pi_1(F_k):F_k\text{ connected component of }F\}\in \sf{T}_S$, and $\pi_1(H)\in \sf{T}_R$. Since $\pi_1(H)$ is a free group, we are in the hypotheses of Theorem \ref{TEOBRGROUPES} and $\dim_{AR}\partial G=1$. 

By Thurston's hyperbolization theorem, $G$ is isomorphic to a discrete, convex cocompact subgroup of $\mathrm{Iso}(\mathbb{H}^3)$. See \cite{KaK} for a similar construction and for more details. On the other hand, since the JSJ splitting of $G$ has a rigid type vertex, $M$ is not a generalized book of $I$-bundles, and therefore, by \cite[Thm 2.9]{CMT}, we have $D(M)>1$. 

\end{document}